\documentclass[10pt, journal, compsoc]{IEEEtran}
\ifCLASSINFOpdf
\else
\fi
\usepackage{amsmath}
\usepackage{bm}
\usepackage{amssymb}
\usepackage{graphicx}
\usepackage{breqn}
\usepackage{amsthm}
\usepackage{etoolbox}
\usepackage{amsfonts}
\usepackage{blindtext}
\usepackage{array}

\usepackage[english]{babel}
\begin{document}
\vspace{250 px}
\title{Quantification of Differential Information using Matrix Pencil}

\author{\IEEEauthorblockN{Snigdha Bhagat\IEEEauthorrefmark{1}, 
S.D.Joshi\IEEEauthorrefmark{2}
}
\IEEEauthorblockA{\IEEEauthorrefmark{1}Indian Institute of Technology, Delhi, India}
}

\maketitle 
\IEEEdisplaynontitleabstractindextext
\IEEEpeerreviewmaketitle
\newtheorem{theorem}{Theorem}%
\AtEndEnvironment{theorem}{\null\hfill\qedsymbol}%
\newcommand{\e}[1]{{\mathbb E}\left[ #1 \right]}
\newcommand*{\QEDA}{\hfill\ensuremath{\blacksquare}}%
\section{abstract}

Any traditional classification problem in general involves modelling individual classes and in turn classification by evaluating the similarity of the test set with the modelled classes. In this paper, we introduce another approach that would find the differential information between two classes rather than modelling individual classes separately. The classes are viewed on a common frame of reference in which one class would have a constant variance, unlike the other class which would have unequal variance along its basis vectors which would capture the differential information of one class over the other.This, when mathematically formulated, leads to the solution of Matrix Pencil equation.The theory of binary classification was extended to a multi-class scenario.This is borne out  by illustrative examples on the classification of the MNIST database.

\section{Introduction}
In general, a classification problem can be solved by modelling individual classes or extracting different kinds of features by analysing the statistical properties of the sample points.  The feature set of individual classes would either form a cluster or adhere to a certain kind of distribution.  The classification of the test set can then be carried out in accordance to its relationship with the feature set of individual classes which can be computed using an appropriate loss function. In this paper, we have attempted to discern one pattern from the other by finding the differential information instead of modelling individual classes. Classification Problem has been solved using several techniques,  what makes Matrix Pencil unique is its ability to discern one pattern, class or process from another by accentuating the differential part. In a binary classification problem both the classes are viewed on a common frame of reference, such that one class would have constant power along all the directions and the other would have unequal power along the modified basis.This unequal power would serve as a metric for quantifying the differential information of one class over the other. This, when formulated, led to the solution of the matrix pencil equation. The motivation of this work came from the analysis of characteristic equation of eigen decomposition $det(A-\lambda I) = 0$ which would in a sense view the information of a certain class represented by A with respect to the standard basis.  The characteristic equation of matrix pencil $det(A-\lambda B) = 0$, on the other hand, can be considered as viewing the information of one class over the other where A and B can be considered as the covariance matrices of the two classes $C_{1}$ and $C_{2}$ respectively.
\par Matrix Pencil has a variety of applications in several other domains some of which are listed below. It is an eigenvalue based approach which is used for estimation of the frequencies, damping factors and other parameters of the exponentially damped/undamped sinusoids. Other applications include Radar Target Classification \cite{5} which models the complex electromagnetic field when a target is incident by an Ultra-Wideband Signal by extraction of features using Matrix Pencil in the time domain and frequency domain.  The features are a function of the natural response of the system which is modelled by the matrix pencil equation.   It has been used in quantum entanglement classification due to its ability to bring about SLOCC equivalence for $2 \bigotimes m \bigotimes n$ systems.  Two states are said to be SLOCC equivalent if they can be transformed into one another with a non zero probability using only local quantum resources and classical communication (SLOCC).  Each state in the $2 \bigotimes m \bigotimes n$ system is represented using the matrix pencil operator, states which are related by invertible transforms are said to be equivalent \cite{1}.   It has also been used for the modal analysis of medical percussion signals due to its inherent ability to decompose the exponentially damped/undamped sinusoids in terms of its components since it is lesser sensitive to noise as compared to Fast Fourier transform, Prony's method and other spectral analysis methods \cite{6}.  Fourier series,  which represents the signal as a sum
of infinite periodic harmonics,  is not naturally suited for decomposition
of short and aperiodic percussion signals since damping would broaden the spectral peaks which in turn would mask the peaks with low amplitude.   This issue is addressed by using the characteristic equation of matrix pencil for decomposing the components of exponentially damped sinusoids\cite{2}.  Matrix Pencil has also been used in literature for the estimation of natural response from multiple transient responses recorded from different Look directions which leads to the estimation of SEM (Singularity Expansion Method)
Poles.    Matrix Pencil method has been employed  for the
simultaneous estimation of all the SEM poles from multiple
look directions without averaging the transient responses obtained\cite{3}.   The applications of Matrix Pencil can also be seen in Speech Processing due to its ability to estimate the signal spectrum at a high resolution due to which it has been used in Speech Enhancement, Compression and Pitch estimation.  Recently Matrix Pencil method was used to predict the time domain response of Reverberation Chamber.  The response of the chamber can be modelled as a linear sum of complex damped sinusoids the parameters of which can easily be found out by solving the matrix pencil equation \cite{4}.   The matrix pencil method has also been employed to synthesize the far-field pattern of a sparse
and random antenna array \cite{7}\cite{8}.   The method starts by sampling the desired pattern from the
original periodic antenna array to form a Hankel matrix,  then singular value decomposition
(SVD) of this matrix is employed to obtain a reduced number of antenna elements,  finally, the matrix pencil method is utilized to reconstruct new element positions and excitations.   On the same lines, matrix pencil has been used for the design of 1-Dimensional or 2-Dimensional metalens with meta-atoms randomly distributed.  Matrix Pencil technique assists the designing of the position and phase of meta-atoms for a random metasurface.  The 1-Dimensional metasurface can be regarded as a linear antenna array whose far-field radiation pattern is a sum of complex exponentials.  Matrix Pencil is used to find the location and phase of meta-atoms keeping the far-field radiation into consideration. Despite several applications of matrix pencil as stated above its usage to discern one pattern over the other has never been exploited.This paper highlights the usage of Matrix Pencil to find the relative information of one pattern over the other.

\section {Matrix Pencil for capturing Differential Information}
Let us consider a binary classification problem with classes denoted as $C_{1}$ and $C_{2}$.   Let X and Y be the vector random variables in N-dimensional space of classes $C_{1}$ and $C_{2}$ respectively,  consisting of their observation vectors. Further, let the covariance matrices of the two classes be denoted as B and A respectively as given in equation \ref{equation 1} and \ref{equation 2} below.  
\newline
\begin{equation}
\label{equation 1}
      B  =  \mathop{\mathbb{E}}((\mathbf{X}-\bm{\mu}_{x}){(\mathbf{X}-\bm{\mu}_{x})}^\dagger)
\end{equation}

\begin{equation}
\label{equation 2}
    A  =  \mathop{\mathbb{E}}((\mathbf{Y}-\bm{\mu}_{y}){(\mathbf{Y}-\bm{\mu}_{y})}^\dagger),
\end{equation}
where $\mu_{x}$ and $\mu_{y}$ are the mean vectors of random vector variables X and Y respectively, $\mathop{\mathbb{E}}$ denotes the expectation operator and the symbol ' $ \dagger$ ' denotes the conjugate transpose of a matrix/vector. We begin with a theorem that would be useful in the latter part of the paper. Without loss of generality, we can assume $\mu_{x}$ and $\mu_{y}$ to be zero vectors thus orthogonality and uncorrelatedness have been used interchangeably.
\begin{theorem}\renewcommand{\qedsymbol}{}
 Let \{\(v_{1}, v_{2}.  .  .  .  .  v_{n}\)\} be any set of n Linearly Independent vectors in $k^n$, where k = $\mathbb{R}$ or $\mathbb{C}$ that is the field of real and complex numbers.  Then there always exists an inner product in which the $\{{v_{i}'s,  1 \leq i \leq n}\}$ are orthogonal/orthonormal.  
\end{theorem}
\begin{proof}
Let us define an linear operator P such that $P v_{i} = e_{i} $ and $\{{e_{i},  1 \leq i \leq n}\}$ is the standard orthonormal basis.   Since  $\{{e_{i},  1 \leq i \leq n}\}$ is an orthonormal set of vectors with respect to standard inner product i.  e.   $ \langle x, y \rangle = \sum\limits_{n} x_{i} y_{i}^{*} $. Thus the proof of the theorem can be stated as below:
\begin{equation}
\label{equation 3}
\begin{split}
 \langle e_{i}, e_{j}\rangle &= \delta_{i, j} \\
\langle e_{i}, e_{j}\rangle &= \langle Pv_{i}, Pv_{j}\rangle \\
\end{split}
\end{equation}
\begin{equation}
\label{equation 4}
\begin{aligned}
\langle Pv_{i}, Pv_{j}\rangle  &= \langle v_{i}, P^{*} P v_{j}\rangle\\
 &= \langle v_{i}, B v_{j}\rangle \\
 &= \langle v_{i},  v_{j}\rangle _{B} \\
 \end{aligned}
\end{equation}
Thus we can say
\begin{equation}
\label{equation 5}
\langle Pv_{i}, Pv_{j}\rangle  = \delta_{i, j}
\end{equation}
Here $P^{*}$ denotes the adjoint of operator P.   Since P is a transformation which maps one basis to another so it is a vector space isomorphism and it would be invertible.   Also $B = P^{*}P$ would be a Hermitian and positive definite matrix which can also be used as the weight matrix and inner product can now be defined as $ \langle x, y \rangle_{B} =  \langle x, By \rangle$, where the inner product on the right-hand side is the standard inner product.  
\end{proof}
We now focus on finding the eigen decomposition of the covariance matrix obtained in equation 1 and 2 which is also the Karhunen Loeve basis. This will be the orthonormal basis in which the coefficients of respective classes would be uncorrelated. The random variable X corresponding to class $C_{1}$ can be represented in terms of eigen vector as given by the below stated equation.  
\begin{equation}
\label{equation 6}
X = \sum\limits_{k} a_k\varphi_k    
\end{equation}
Where $a_{k} = \langle X, \phi_{k} \rangle $ and  $\{{a_{k},  1 \leq k \leq n}\}$ are uncorrelated.  
Rewriting the above equation using the eigen values ($ \lambda_k $) of X we get:
\begin{equation}
\label{equation 7}
X = \sum\limits_{k} (\frac{a_k}{\sqrt{\lambda_k}})(\sqrt{\lambda_k}\varphi_k)
\end{equation}

The modified basis and their corresponding coefficients can be termed as scaled coefficient  $ \tilde a_k =  \frac{a_k}{\sqrt{\lambda_k}} $  {and scaled basis  } $u_k = \sqrt{\lambda_k}\varphi_k  $.   The covariance matrix A and B are positive definite matrices and so all the eigen values would be positive i.e. $\{{\lambda_{k} > 0 ,  1 \leq k \leq n}\}$ . Rewriting equation \ref{equation 7} in terms of the modified basis leads to the following representation of X.  
\begin{equation}
\label{equation 8}
X = \sum\limits_{k} {\tilde a_k} u_k
\end{equation}
It is observed that in the modified representation of X the basis is orthogonal but not orthonormal and the variance of all the projection coefficients $\{{\tilde a_{k},  1 \leq k \leq n}\}$ is the same in all directions and is equal to unity which implies that X will have equal power along the basis $\{{u_{k},  1 \leq k \leq n}\}$. Thus, for any vector random variable, there will always exist a basis such that the projection coefficients would have equal variance along all the directions and the basis need not be orthonormal. Since equation \ref{equation 6} is the solution of Karhunen Loeve Transform (KLT) we know that  $\mathop{\mathbb{E}}[a_{k}\overline a_{j}] = \lambda_{k}\delta_{k, j}$.  The variance of the scaled coefficient  $\{ {\tilde a_k ,  1 \leq k \leq n}\}$ can be written as below:
\begin{equation}
\label{equation 9}
\begin{aligned}
\mathop{\mathbb{E}}[\tilde a_{k} {\tilde a_{k}}^{*}] &= \mathop{\mathbb{E}}[\frac{a_k}{\sqrt{\lambda_k}}][\frac{{a_k}^{*}}{\sqrt{\lambda_k}}]  \\
\mathop{\mathbb{E}}[\tilde a_{k} {\tilde a_{k}}^{*}] &= \frac{1}{\lambda_k} \mathop{\mathbb{E}}[a_{k}{a_k^{*}}] \\
\mathop{\mathbb{E}}[\tilde a_{k} {\tilde a_{k}}^{*}] &= \frac{1}{\lambda_k} \lambda_{k}\delta_{k} \\
\mathop{\mathbb{E}}[\tilde a_{k} {\tilde a_{k}}^{*}] &= \delta_{k} = 1
\end{aligned}
\end{equation}
The above discussion shows that for any random vector we can always define a basis and an inner product such that the same vector when represented in modified basis would have equal variance of projection coefficients $\{{\tilde a_{k},  1 \leq k \leq n}\}$ and an orthonormal basis with respect to the modified inner product.  In order to obtain such orthonormal basis for the orthogonal basis$\{{u_{k},  1 \leq k \leq n}\}$ we introduce the operator L such that:
\begin{equation}
\label{equation 10}
L\sqrt{\lambda_k}\varphi_k  = e_k,
\end{equation} 
where $\{{e_{k},  1 \leq k \leq n}\}$ are the standard basis vectors and $\lambda_k$'s are the eigen values corresponding to eigen vectors $\varphi_k$.  
Now consider $LX$ which is the transformed version of X such that its projection coefficients have equal variance in all the directions and the basis is standard orthonormal eigen basis $e_{k} $.Applying the operator L on both sides of \ref{equation 8} we get:  
\begin{equation}
\label{equation 11}
\begin{aligned}
&LX =   \sum\limits_{k}  \tilde a_k (L u_k), \\
&LX = \sum\limits_{k}  \tilde a_k (e_k),
\end{aligned}
\end{equation}
where $\{{e_{k},  1 \leq k \leq n}\}$ is the standard basis.  
Using equation \ref{equation 10},the operator matrix L can be obtained as follows:
\begin{equation}
\label{equation 12}
L = \begin{bmatrix}\lambda_{1}^{-1/2} & & \\ & \ddots & \\ & & \lambda_{n}^{-1/2} \end{bmatrix} \begin{bmatrix}\varphi_{1}^{\dagger} \\ \varphi_{2}^{\dagger} \\ \vdots \\ \varphi_{n}^{\dagger} \end{bmatrix}
\end{equation}

\begin{equation}
\label{equation 13}
L = \Lambda^ {-1/2} \varphi^{\dagger}
\end{equation}
The operator matrix L, derived from class $C_{1}$, would perform the operation of whitening the signal X such that the power is equally distributed in all the directions. The feature set of class $C_{1}$ would now be whitened and have equal variance of projection coefficients in the scaled eigen basis frame of reference. The feature set of class $C_{2}$ when operated by the same operator matrix i.e. $LY$ would have an unequal variance of projection coefficients in this scaled eigen basis. Thus viewing the feature vector of both the classes on the same frame of reference would enable us to encapsulate the differential information of class $C_{2}$ with respect to class $C_{1}$ in the form of unequal projection coefficients over the uniformly distributed projection coefficients of class $C_{1}$. 

As stated in Theorem 1,for any set of linearly independent vectors there would always exist an inner product in which the vectors would be orthogonal/orthonormal.  As evident in equation \ref{equation 8}, for the linearly independent vectors $\{{u_{k},  1 \leq k \leq n}\}$, there would exist a modified inner product in which they are orthonormal. The weight matrix of this modified inner product is obtained in equation \ref{equation 17} below.
\begin{equation}
\label{equation 14}
\begin{split}
 \langle e_{i}, e_{j}\rangle = \delta_{i, j} \\
\langle e_{i}, e_{j}\rangle = \langle Lu_{i}, Lu_{j}\rangle \\
\end{split}
\end{equation}
\begin{equation}
\label{equation 15}
\langle Lu_{i}, Lu_{j}\rangle  = \langle u_{i}, L^{\dagger} L u_{j}\rangle
\end{equation}
The value of $L^{\dagger} L$ is evaluated to find the value of weight matrix.  
\begin{equation}
\label{equation 16}
L^{\dagger} L = \varphi\Lambda^ {-1/2}  \Lambda^ {-1/2} \varphi^{\dagger} = B^{-1}
\end{equation}
\begin{equation}
\label{equation 17}
\begin{aligned}
\langle u_{i}, L^{\dagger} L u_{j}\rangle  &= \langle u_{i}, B^{-1} u_{j}\rangle \\
 \langle e_{i}, e_{j}\rangle &= \langle u_{i}, u_{j}\rangle_{B^{-1}} \\
\end{aligned}
\end{equation}
At this point we can define the weight matrix $W =  B^{-1}$ such that $\langle u_{i}, u_{j}\rangle_{W} = \delta_{i, j} $.   $\{{u_{k},  1 \leq k \leq n}\}$ will be orthonormal and the projection coefficients$\{{\tilde{a_{k}},  1 \leq k \leq n}\}$ will be uncorrelated white under the action of modified inner product.   
\newline
\begin{theorem}\renewcommand{\qedsymbol}{}
Let \(\mathbf{x_{1}, x_{2}.  .  .  .  .  x_{n}}\) be a set of linearly independent random variables such that $X \equiv$ \(\mathbf{[x_{1}, x_{2}.  .  .  .  .  x_{n}]^{T}} \).  Given the KLT representation of X we can obtain a whitening operator L such that $L\sqrt{\lambda_k}\varphi_k  = e_k$ where  $\{{e_{k},  1 \leq k \leq n}\}$  are the standard basis vectors and $LX$ is the whitened version of $X$.  
\end{theorem}
\begin{proof}
In order to prove that LX is the whitened version of X we need to prove that the auto correlation of the signal is identity.   
\begin{equation}
\label{equation 18}
\begin{aligned}
\mathop{\mathbb{E}}[(LX)(LX)^\dagger] &= I \\
\mathop{\mathbb{E}}[(LX)(LX)^\dagger] &= {\mathbb{E}}[LX{X^\dagger}{L^\dagger}] \\
\mathop{\mathbb{E}}[(LX)(LX)^\dagger] &=  {L\mathbb{E}}[X{X^\dagger}]{L^\dagger}\\
\end{aligned}
\end{equation}
The eigen decomposition of the covariance matrix of class $C_{1}$ can be given as $ B = \mathop{\mathbb{E}}(X X^\dagger) = \varphi\Lambda\varphi^{\dagger}$. Substituting the values of L from equation \ref{equation 13} in equation \ref{equation 18}, we obtain -  
\begin{equation}
\label{equation 19}
\begin{aligned}
\mathop{\mathbb{E}}[(LX)(LX)^\dagger] &= L \mathop{\mathbb{E}}(X X^\dagger) (L^\dagger) \\
\mathop{\mathbb{E}}[(LX)(LX)^\dagger] &=  [\Lambda^ {-1/2} \varphi^{\dagger} (\varphi\Lambda\varphi^{\dagger})\varphi(\Lambda)^ {-1/2}] 
\end{aligned}
\end{equation}
Since $[\varphi]$ is an orthonormal basis $   \varphi \varphi^{\dagger} = \varphi^{\dagger}\varphi = I $.  
\begin{equation}
\label{equation 20}
\mathop{\mathbb{E}}[(LX)(LX)^\dagger] = \Lambda^ {-1/2}\Lambda\Lambda^ {-1/2} = I
\end{equation}
\end{proof}
Now we prove a result which is going to play a fundamental role in the theory developed latter.Thais result although well known is proved here in the context of the problem statement.
\begin{theorem}\renewcommand{\qedsymbol}{}
If a random vector Z is i. i. d. white in a particular orthonormal frame of reference then it will be i. i. d. white in any other orthonormal frame of reference.  
\end{theorem}
\begin{proof}
Let Z be a random vector and its KLT representation can be given by
\begin{equation}
\label{equation 21}
Z = \sum\limits_{k} {a_k} \varphi_{k},
\end{equation}
where $\{{\varphi_{k},  1 \leq k \leq n}\}$ is the orthonormal eigen basis and $\e{a_k{a_l^*}} = \sigma^{2}\,\,\delta_{k,l}\,\,\forall k $. Let Z be represented by another set of orthonormal basis function $( \psi_{k} )$ as given below.  
\begin{equation}
\label{equation 22}
Z = \sum\limits_{k} {b_k} \psi_{k}
\end{equation}
We need to prove $\e{b_k{b_p^*}} = \sigma^{2}\,\,\delta_{k,p}\,\,\forall k $. 
Let us at this point define a linear operator A such that $A \varphi_{k} = \psi_{k} $.  Since  A maps an orthonormal basis $\{{\varphi_{k},  1 \leq k \leq n}\}$ to another orthonormal basis $\{{\psi_{k},  1 \leq k \leq n}\}$ by definition it will be a unitary matrix. 
\begin{equation}
    \label{equation 23}
    b_k = \langle Z, \psi_{k}\rangle=\sum\limits_{n}{Z(n)}\psi_{k}^*(n) = Z^{T}{\psi_{k}^*}
\end{equation}
Since $b_{k}$ is a scalar quantity we can write $b_{k}^{T} = b_k$ i.e.
\begin{equation}
    \label{equation 24}
    b_k = Z^{T}{\psi_{k}^*} = {\psi_{k}^\dagger}Z
\end{equation}
\begin{equation}
\label{equation 25}
\begin{split}
&\e{b_k {b_p^*}} = \e{({\psi_{k}^\dagger}Z) (Z^{\dagger}\psi_{p})} \\
&\e{b_k {b_p^*}} = {\psi_{k}^\dagger} \e{ Z Z^{\dagger}}\psi_{p} \\
&\e{b_k {b_p^*}} = {\varphi_{k}^\dagger} A^\dagger\,\ (\sigma^{2} \delta_{k,p})\,\ A \varphi_{p}\\
&\e{b_k {b_p^*}} =  (\sigma^{2} \delta_{k,p}) \,\ {\varphi_{k}^\dagger} A^\dagger A \varphi_{p}
\end{split}
\end{equation}
Since A is a unitary matrix $AA^{\dagger} = A^{\dagger}A = I$.  
\begin{equation}
\label{equation 26}
\begin{aligned}
 &\e{b_k {b_p^*}} =  \sigma^{2} \delta_{k,p} \,\ {\varphi_{k}^\dagger}\varphi_{p} \\ 
  &\e{b_k {b_p^*}} = \sigma^{2} \delta_{k,p}
\end{aligned}
\end{equation}
\end{proof}
With all the preliminaries established we now discuss the main result of the paper.
\begin{theorem}\renewcommand{\qedsymbol}{}
Let B and A be the covariance matrices of classes $C_{1}$ and $C_{2}$ , respectively, as denoted earlier. Further let L be the whitening operator corresponding to class $C_{1}$.Then $\{{\psi_{k},  1 \leq k \leq n}\}$ is the Karhunan-Loeve basis of LY if and only if $\{{\tilde \psi_{k},  1 \leq k \leq n}\}$) is a solution of the matrix pencil equation.  
\end{theorem}
\begin{proof}
Let L be the whitening operator corresponding to class $C_1$ as defined earlier. The KL decomposition of LY can be written as follows:
\begin{equation}
\label{equation 29}
LY = \sum\limits_{k}{c_k}\psi_{k},
\end{equation}
where $\{{c_{k},  1 \leq k \leq n}\}$ are uncorrelated and $\{{\psi_{k},  1 \leq k \leq n}\}$ is the orthonormal eigen basis corresponding to covariance matrix of LY given as:
\begin{equation}
\label{equation 27}
cov_{LY} = E(LYY^{\dagger}L^{\dagger}) \\
 = LAL^{\dagger}.
\end{equation}
Denoting the eigen vector of $LAL^{\dagger}$ corresponding to its eigen value $\mu_{k}$, as  $\psi_{k}$ we can write:  
\begin{equation}
\label{equation 28}
LAL^{\dagger}\psi_{k} = \mu_{k} \psi_{k} 
\end{equation}
\newline Denoting $ \tilde \psi_{k} = L^{\dagger}\psi_{k} $ and algebraically simplifying equation \ref{equation 28}   we obtain:

\begin{equation}
\label{equation 30}
\begin{aligned}
&LAL^{\dagger}\psi_{k} = \mu_{k} \psi_{k} \\
\iff &LA \tilde \psi_{k}  =\mu_{k} \psi_{k} \\
\iff &LA \tilde \psi_{k} = \mu_{k} (L^{\dagger})^{-1}L^{\dagger}\psi_{k} \\
\iff &LA \tilde \psi_{k} = \mu_{k} (L^{\dagger})^{-1}\tilde \psi_{k} \\
\iff &A \tilde \psi_{k} = \mu_{k} L^{-1}(L^{\dagger})^{-1}\tilde \psi_{k} \\
\iff &A \tilde \psi_{k} = \mu_{k} (L^{\dagger}L)^{-1}\tilde \psi_{k} \\
\end{aligned}
\end{equation}
Plugging this result derived in equation \ref{equation 16} in equation \ref{equation 30} we get:
\begin{equation}
\label{equation 31}
A \tilde \psi_{k} = \mu_{k}B\tilde \psi_{k} 
\end{equation}
\end{proof}
At this point it can be noted that $\tilde \psi_{k}$ are orthonormal in the modified inner product.This has been proved by the below stated equations.
\begin{equation}
\label{equation 32}
\begin{aligned}
 \langle\tilde \psi_{k},\tilde \psi_{p}\rangle_{B} &=  \langle\tilde \psi_{k},B \tilde \psi_{p}\rangle \\
 &= \langle L^{\dagger}\psi_{k},B L^{\dagger} \psi_{p}\rangle \\
 &= \langle \psi_{k}, L B L^{\dagger} \psi_{p}\rangle \\
 \end{aligned}   
\end{equation}
Let us evaluate the value of $L B L^{\dagger}$:
\begin{equation}
\label{equation 33}
L B L^{\dagger} = \Lambda^ {-1/2} \varphi^{\dagger}(\varphi \Lambda \varphi^{\dagger})\varphi \Lambda^ {-1/2} = I \\
\end{equation}
\begin{equation}
\label{equation 34}
\begin{aligned}
\langle\tilde \psi_{k},\tilde \psi_{p}\rangle_{B} &= \langle \psi_{k},\psi_{p}\rangle \\
\langle\tilde \psi_{k},\tilde \psi_{p}\rangle_{B} &= \delta_{k,p}
\end{aligned}
\end{equation}
Thus it can be stated that the KL basis of LY i.e. $\{{\psi_{k},  1 \leq k \leq n}\}$ is orthogonal with respect to the standard inner product if an only if the transformed basis $\{{\tilde \psi_{k},  1 \leq k \leq n}\}$, which form the solution to the matrix pencil equation, is orthogonal with respect to the modified inner product.
\par
The above section began with the random feature vectors $X$ and $Y$ of classes $C_{1}$ and $C_{2}$ respectively. The feature vectors of class $C_{1}$ were expressed in terms of its KL basis as given in equation \ref{equation 5}, transforming them into a set of uncorrelated coefficients $\{{a_{k},  1 \leq k \leq n}\}$.  The same feature vector X was then represented in terms of a modified basis $\{{u_{k},  1 \leq k \leq n}\}$ where the coefficient $\{{\tilde a_{k},  1 \leq k \leq n}\}$ have unity variance in all directions. Observe that the basis here is orthogonal but not orthonormal. Since the representation of a vector in terms of a different basis can also be effected by a transformation of space we defined an operator L which would serve the purpose.This also amounts to considering LX rather than X. Since $X$ is transformed to $LX$ which is i.i.d. in the standard eigen basis. By Theorem 3, it would be i.i.d. white in any other orthonormal frame of reference. In particular it would also be i.i.d. white in the KL basis of LY i.e. $\{{\psi_{k},  1 \leq k \leq n}\}$. The coefficients of $LY$ in terms of its KL basis would be uncorrelated but would not have equal variance along all the directions. So the quantification of non uniform variance can be used to measure the differential information of one class over the other. So the variation of $\{{c_{k},  1 \leq k \leq n}\}$ which are the coefficients of projection along the KL basis of $LY$  (equation \ref{equation 29}) over the coefficients of  $\{{\tilde a_{k},  1 \leq k \leq n}\}$ which are uniformly distributed (equation \ref{equation 8}) can provide a metric of extra information of one class with respect to the other. The unequal variance of projection coefficients in different directions over unity variance of projection coefficients would essentially capture the differential information amongst the two classes.
\par In order to validate the theory proposed, we demonstrate through various experimentation's that the eigen vectors of the Matrix Pencil $ A - \lambda B $ quantify the differential information between two classes. It is observed through experimental studies that these eigen vectors, when augmented with the eigen vectors of the reference class ($C_{1}$ with covariance matrix B in this case), classify the two classes with reasonable accuracy. In the next section we present the simulation studies.

\section{Experimentation}
This section presents simulation results to validate the above proposed theory. The simulation study has been divided into three parts (i) Binary classification (ii) Multi-class classification and (iii) Transformation of one pattern to another. Whereas the first two parts are designed to demonstrate the efficacy in classification problem the third part presents a spin off in transforming one pattern to another.
\subsection{Binary Classification using Matrix Pencil}

The experimentation has been carried out to solve the binary classification problem for MNIST (Modified National Institute of Standards and Technology Database). The random feature vectors X and Y, of classes $C_{1}$ and $C_{2}$, respectively, have been obtained by vectorising the input images, and a standard K-Nearest Neighbour (k-NN) classifier has been employed for classification. Let us, for instance, take the matrix pencil having the characteristic equation $(A- \lambda B)x = 0 $ , this would view class $C_2$ with covariance matrix $A$ over and above class $C_{1}$ with covariance matrix $B$, thus class $C_1$ in this case can be viewed as the reference class. In the theory, presented in the above sections, we have proposed that matrix pencil helps quantify the differential information of one class over the other. Thus we propose to consider a modified feature set which would be a set of projection coefficients along the eigen basis of matrix pencil augmented with the eigen vectors of the reference class. The eigen vectors of the matrix pencil are obtained as given in equation \ref{equation 35} - \ref{equation 37}. The feature set, denoted by $ (A-\lambda B; B)$, is a vector comprising of projection coefficients along the eigen vectors of matrix pencil and the reference class. Simulation is also carried on different feature sets i.e. projection coefficients along the eigen vectors of matrix pencil of the form $A - \lambda B$ or $B - \lambda A$. The classification accuracy (in percentage) with different feature sets is tabulated in Table \ref {Table 1}.

\subsection{A note on Simulation Results}
\par In order to make the simulations carried out more explicit the eigen vectors of the matrix pencil are derived below.Let us begin with the characteristic equation of matrix pencil. 
 
\begin{equation}
\label{equation 35}
\begin{aligned}
det(&A-\lambda B) = 0 \\
\iff det(&A - \lambda \phi^{-1}\Lambda\phi) = 0 \\
\iff det(&\phi^{-1}(\phi A \phi^{-1}-\lambda\Lambda)\phi) = 0 \\
\iff det(&\phi^{-1}) det(\phi A \phi^{-1}-\lambda\Lambda) det(\phi) = 0 \\
\end{aligned}
\end{equation}
For simplification we define \(\Lambda =N^{2}\). Since $\phi$ is the eigen vector matrix of class $C_{2}$ $det(\phi) = det(\phi^{-1}) \neq 0 $.Thus we can write:
\begin{equation}
\label{equation 36}
\begin{aligned}
&det(\phi A \phi^{-1}-\lambda N N) = 0 \\
\iff &det(N^{-1} \phi A \phi^{-1} N^{-1} - \lambda I ) = 0 \\
\end{aligned}
\end{equation}
Equation \ref{equation 36} is the characteristic equation of standard Eigen value problem.This implies that there exists a $v$ such that $M v = 0$ where $ M = N^{-1} \phi A \phi^{-1} N^{-1} - \lambda I $ .
\begin{equation}
\label{equation 37}
\begin{aligned}
(N^{-1} \phi A \phi^{-1} N^{-1} - \lambda I )v = 0 \\
\iff N^{-1} \phi A \phi^{-1} N^{-1} v = \lambda v \\
\iff A \phi^{-1} N^{-1}v = \lambda \phi^{-1} N v \\
\end{aligned}
\end{equation}
Let $m = \phi^{-1} N^{-1}v$, substituting this value in \ref{equation 37} we obtain:
\begin{equation}
\begin{aligned}
    &Am = \lambda \phi^{-1} N (N \phi m) \\
    \iff &Am = \lambda (\phi^{-1} \Lambda \phi) m \\
    \iff &Am = \lambda B m \\
    \end{aligned}
\end{equation}
Thus we obtain back the standard matrix pencil equation which proves that $m = \phi^{-1} N^{-1}v$ are the eigen vectors of the Matrix Pencil. 
\begin{table}[h]
\caption{Classification accuracy (in percentage) for binary classification} 
\setlength\tabcolsep{1.5pt} 
\begin{center}
\begin{tabular}{| c | c | c | c | c | c|} 
\hline
\label{Table 1}
 $C_1$ & $C_2$ & $(A-\lambda B;B)$  & $(B-\lambda A;A)$ &  $A-\lambda B$ & $B-\lambda A$\\ \hline
1 & 0 & 99.  90 & 99.  90 & 64.  77 & 90.  50 \\ 
2 & 0 & 99.  35 & 99.  45 & 57.  85 & 72.  16 \\ 
3 & 0 & 99.  90 & 99.  80 & 64.  97 & 65.  17 \\ 
4 & 0 & 99.  80 & 99.  85 & 55.  61 & 80.  73 \\ 
5 & 0 & 99.  46 & 99.  57 & 59.  61 & 62.  77 \\ 
6 & 0 & 99.  38 & 99.  38 & 58.  51 & 51.  50 \\ 
7 & 0 & 99.  95 & 99.  85 & 60.  90 & 87.  75 \\ 
8 & 0 & 99.  13 & 99.  59 & 51.  89 & 68.  57 \\ 
9 & 0 & 99.  30 & 99.  49 & 51.  03 & 79.  94 \\ 
2 & 1 & 99.  35 & 99.  81 & 89.  85 & 81.  96 \\ 
3 & 1 & 99.  81 & 99.  86 & 76.  60 & 85.  55 \\ 
4 & 1 & 99.  62 & 99.  53 & 69.  77 & 88.  62 \\ 
5 & 1 & 99.  85 & 99.  95 & 75.  77 & 86.  98 \\ 
6 & 1 & 99.  71 & 99.  86 & 87.  48 & 84.  85 \\ 
7 & 1 & 98.  61 & 99.  35 & 76.  28 & 80.  44 \\ 
8 & 1 & 99.  76 & 99.  95 & 84.  73 & 89.  24 \\ 
3 & 2 & 99.  46 & 99.  61 & 73.  41 & 72.  87 \\ 
4 & 2 & 99.  95 & 99.  85 & 74.  63 & 86.  59 \\ 
5 & 2 & 99.  95 & 100 & 73.  70 & 70.  22   \\ \hline
\end{tabular}
\end{center}
\end{table}
\begin{table*}[t!]
\caption{Classification accuracy (in percentage) for Multi-class Classification}
\label{Table 2}
\begin{center}
\begin{tabular}{| m{1em} | m{1em} | m{1em} | m{5.3em}  | m{7.5em} | m{9.5em} | m{5.3em}  | m{7.5em} | m{9.5em} | }
  \hline
$C_{1}$ & $C_{2}$ & $C_{3}$ & $A-\lambda (B,C)$ & $[A- \lambda (B,C)]; A$ & $[A-\lambda (B,C)]; (B,C)$ & $B-\lambda (A,C)$ & $[B-\lambda (A,C)]; B$ & $[B-\lambda (A,C)]; (A,C)$\\
\hline
0 & 2 & 8 & 37.27 & 70.06 & 98.66 & 36.63 & 70.06 & 98.66 \\ 
0 & 3 & 5 & 36.05 & 59.33 & 98.16 & 32.75 & 59.33 & 98.16 \\ 
0 & 3 & 6 & 41.75 & 69.40 & 99.32 & 50.20 & 69.40 & 99.32 \\ 
1 & 4 & 8 & 47.78 & 46.00 & 99.19 & 56.20 & 46.00 & 99.19 \\ 
1 & 4 & 9 & 66.44 & 46.26 & 97.98 & 59.09 & 46.26 & 97.98 \\ 
3 & 5 & 8 & 38.14 & 59.42 & 96.49 & 34.91 & 59.42 & 96.49 \\ 
3 & 5 & 9 & 49.64 & 72.00 & 97.32 & 34.76 & 72.00 & 97.32 \\ 
3 & 6 & 7 & 44.76 & 87.72 & 99.43 & 39.19 & 87.72 & 99.43 \\ 
3 & 6 & 8 & 35.15 & 71.04 & 98.20 & 32.22 & 71.04 & 98.20 \\ 
4 & 7 & 9 & 59.16 & 62.24 & 96.99 & 35.64 & 62.24 & 96.98 \\ 
4 & 8 & 9 & 32.21 & 52.51 & 96.76 & 34.94 & 52.51 & 96.76 \\ 
5 & 6 & 7 & 40.69 & 83.46 & 99.34 & 39.23 & 83.46 & 99.34 \\ 
5 & 6 & 8 & 45.01 & 67.88 & 97.84 & 37.89 & 67.88 & 97.84 \\ 
6 & 7 & 8 & 38.75 & 74.90 & 99.19 & 50.57 & 74.90 & 99.19 \\ 
6 & 7 & 9 & 68.41 & 66.51 & 98.50 & 36.53 & 66.51 & 98.49 \\  \hline
\end{tabular}
\end{center}
\end{table*}
\subsection {Multi-class Classification using Matrix Pencil}

The theory for binary classification was extended for the multi-class classification problem.  In a multi-class classification problem let us consider three class $C_{1}, C_{2}, C_{3}$ whose covariance matrices are defined by $A$, $B$ and $C$. The differential information of one class with respect to other classes is found using the eigen vectors of the matrix pencil $A - \lambda (B,C)$ or $B-\lambda (A,C)$ or $C-\lambda (A,B)$.  $A - \lambda (B,C)$ would find the additional information of class $ C_{1} $ whose covariance matrix is given by A with respect to classes $C_{2}$ and $C_{3}$, the sample points of classes  $C_{2}$ and $C_{3}$ are combined, and the eigen vectors of these combined sample points are denoted as $(B, C)$ which can also be treated as the reference class in our case. In order to solve any multi-class classification problem, the eigen vectors of the matrix pencil are augmented with the eigen vectors of the reference class. For example the eigen vectors of the matrix pencil $A-\lambda (B,C)$ are clubbed with the eigen vectors of the reference class $(B,C)$ which is denoted as $[ A-\lambda (B,C)]; (B,C)$. The proposed theory was tested on various types of feature sets which are projection coefficients along $A-\lambda (B,C)$, $[A-\lambda (B,C)]; (B,C)$, $[A-\lambda (B,C)];A$.The results of multi-class classification for various types of feature sets are summarised in Table \ref{Table 2}. This multi-class classification problem can also be solved by the using the feature set composed of projection coefficients along the eigen basis of $[A- \lambda (B,C)]$ augmented with $ B-\lambda C $ and the reference class $C_3$ with covariance matrix $C$, which can be denoted as $A-\lambda (B,C);B-\lambda C;C$. Classification accuracy has been evaluated for the feature set $A-\lambda (B,C);B-\lambda C;C$ and $A-\lambda (B,C);B-\lambda C$ where in the latter case the projection coefficients along the eigen vectors of the reference class have not been included in the feature set. The results of classification are tabulated in table \ref{Table 3} .

\begin{table}[h]
\caption{Classification accuracy (in percentage) for Multi-class Classification}
\begin{center}
\begin{tabular}{| m{1em} | m{1em} | m{1em} | c | c |}
\hline
\label{Table 3}
$C_{1}$ & $C_{2}$ & $C_{3}$ & $A-\lambda (B,C);B-\lambda C$ & $A-\lambda (B,C);B-\lambda C;C$ \\
\hline
0 & 2 & 4 & 67.20  & 98.80 \\
0 & 4 & 7 & 65.59 & 98.76 \\
1 & 2 & 7 & 78.97 & 95.40 \\
1 & 4 & 5 & 74.74 & 98.77 \\
2 & 3 & 6 & 61.87 & 98.27 \\
2 & 4 & 9 & 60.57 & 96.49 \\
2 & 6 & 7 & 55.43 & 97.42 \\
3 & 4 & 9 & 64.81 & 95.77 \\
3 & 7 & 9 & 40.99 & 97.11 \\ 
\hline
\end{tabular}
\end{center}
\end{table}

\subsection{Transformation of one pattern to another}
This section presents a spin off to the above proposed theory.
The theorem presented in the above section proves that the operator matrix L is a whitening matrix.  Whitening matrices $L_{X}$ and $L_{Y}$ were derived from the classes $C_{1}$ and $C_{2}$.  $L_{X}$ when operated on X would whiten the data point of class 1 as evident from equation \ref{equation 14}. This whitened signal when operated by $L_{Y}^{-1}$, would perform the corresponding coloring operation. It was observed that when operator $L_{Y}^{-1}$ was applied on the whitened signal $L_{X}X$, could transform one pattern corresponding to one class to the pattern corresponding to the second class.  The pattern obtained after performing the transformation of $L_{Y}^{-1}L_{X}X$ is a function of how well $L_{X}$ can whiten X and the generalization of the operator matrix $L_{Y}$ of class $C_{2}$ . Since $L_{X}X$ is a white signal; any pattern can be generated by applying the transformation $L_{Y}^{-1}P$ onto any white signal P such that $L_{Y}$ is the whitening matrix derived from the class of patterns we intend to generate.  The above experiment was validated on MNIST and MPEG-7 Core Experiment CE-Shape-1 database and it was observed that when the feature vector matrices are such that the dimensionality of the feature vector of each sample point is greater than the number of examples the generalization of the operator matrix will not be as good as when the number of examples is greater than the dimensionality of each feature vector.  

\begin{figure}[h!]
\includegraphics[scale=0.5,  width= \linewidth]{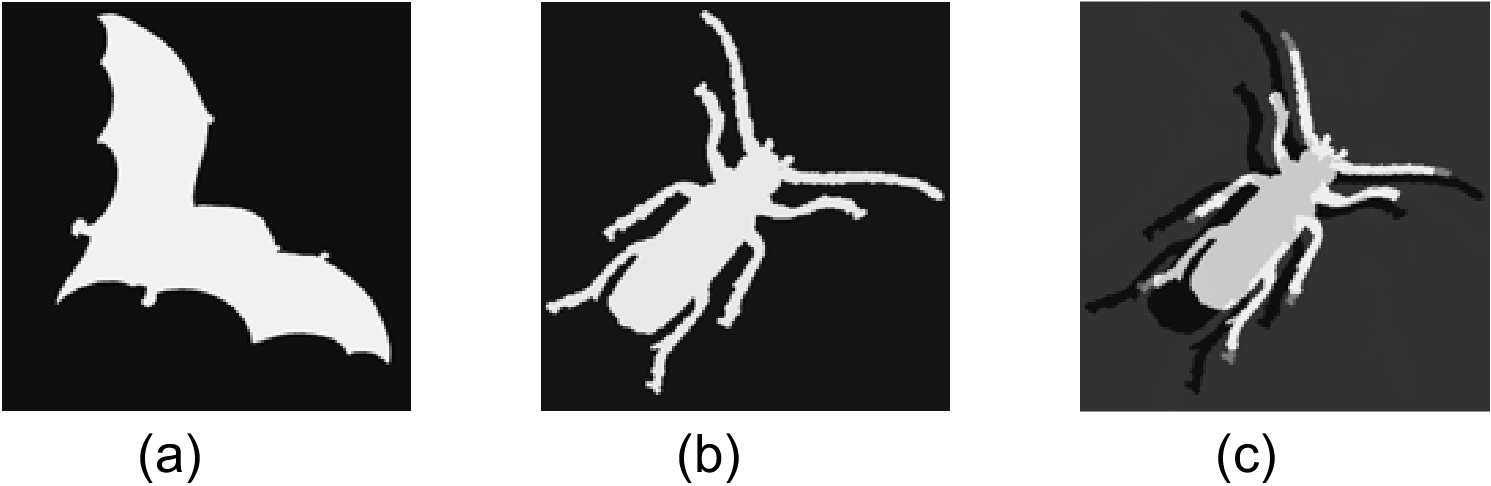}
\caption{(a) Data sample from class 1 (b) Data sample from class 2 (c) Transformation output}
\end{figure}
\section{Conclusions}
In this paper we have presented a theory to find the information that would discern one pattern from another using the characteristic equation of matrix pencil. The efficacy of the proposed scheme has been demonstrated for binary and multi-class classification. And as a spin off the transformation of one pattern to another is also presented.


\begin{thebibliography}{1}


\bibitem{1}
Eric Chitambar, 
Carl A.   Miller, 
Yaoyun Shi
 \emph{Matrix Pencils and Entanglement Classification},  3rd~ed.  \hskip 1em plus
  0.5em minus 0.4em\relax  Volume 51,  Issue 7, 10.  1063/1.  3459069
\bibitem{2}
Moinuddin Bhuiyan∗,  Eugene V.   Malyarenko,  Mircea A.   Pantea,  Fedar M.   Seviaryn, 
and Roman Gr.   Maev
 \emph{Advantages and Limitations of Using Matrix Pencil
Method for the Modal Analysis of Medical
Percussion Signals}, \hskip 1em plus
  0.5em minus 0.4 em\relax  IEEE Transaction on Biomedical Engineering,  VOL.   60,  NO.   2,  February 2010
  \bibitem{3}
Tapan Kumar Sarkar,  Sheeyun Park,  Member,  Jinhwan Koh,  and Sadasiva M.   Rao
 \emph{Application of the Matrix Pencil Method for
Estimating the SEM (Singularity Expansion Method)
Poles of Source-Free Transient Responses from
Multiple Look Directions}, \hskip 1em plus
  0.5em minus 0.4 em\relax  IEEE Transactions on antenna and Propagation,  Vol.   48,  No.   4,  April 2000
   \bibitem{4}
Song Wang,  Zhan C.   Wu,  Lei Du,  Guang H.   Wei,  and Yao Z.   Cui
 \emph{Study on the Matrix Pencil Method with Application to Predict
Time-domain Response of a Reverberation Chamber}, \hskip 1em plus
  0.5em minus 0.4 em\relax  ACES Journal,  Vol.   28,  No.  9,  September 2013
   \bibitem{5}
Mahmoud Khodjet-Kesba,  Khalil El Khamlichi Drissi,  Sukhan Lee, 
Kamal Kerroum,  Claire Faure and Christophe Pasquier \emph{Comparison of Matrix Pencil Extracted Features
in Time Domain and in Frequency Domain for Radar
Target Classification}, \hskip 1em plus
  0.5em minus 0.4 em\relax  Hindawi Publishing Corporation
International Journal of Antennas and Propagation
Volume 2014,  Article ID 930581,  9 pages
http://dx.  doi.  org/10.  1155/2014/930581
\bibitem{6}
Tapan K.   Sarkar and Odilon Pereira
 \emph{Using the Matrix Pencil Method to Estimate the
Parameters of a Sum of Complex Exponentials}, \hskip 1em plus
  0.5em minus 0.4 em\relax  lEEE Antennas and Propagation Magazine,  Vol.   37,  No.   1,  February 1995
  \bibitem{7}
  Y.   Liu,  Z.   Nie,  and Q.   H.   Liu,  \emph{Reducing the number of elements in a linear antenna array by the matrix pencil method}, \hskip 1em plus
  0.5em minus 0.4 em\relax IEEE Trans.   Antenn.   Propag.   56(9),  2955–2962 (2008).  
   \bibitem{8}
  Y.   Liu,  Q.   H.   Liu,  and Z.   Nie,   \emph{Reducing the number of elements in the synthesis of shaped-beam patterns by the
forward-backward matrix pencil method}, \hskip 1em plus
  0.5em minus 0.4 em\relax IEEE Trans.   Antenn.   Propag.   58(2),  604–608 (2010).  
  \bibitem{9}
  Mengmeng Li, Shuaishuai Li, Yefeng Yu, Xingjie Ni and Rushan Chen
\emph{  Design of random and sparse metalens with
matrix pencil method}
\hskip 1em plus
  0.5em minus 0.4 em\relax Vol.   26,  No.   19 | 17 Sep 2018 | Optics Express 24702
  \bibitem{10}
  Marc Jungers,  Cristian Oară,  Hisham Abou-Kandil,  and Radu Ştefan
 \emph{General Matrix Pencil Techniques for Solving Discrete-Time Nonsymmetric Algebraic Riccati Equations}\hskip 1em plus
  0.5em minus 0.4 em\relax SIAM Journal on Matrix Analysis and Applications, Volume 31,  Issue 3

\end{thebibliography}
\end{document}